\documentclass[11pt]{article}

\usepackage{amssymb, amsmath, amsthm, graphicx}
\usepackage[left=1in,top=1in,right=1in]{geometry}

\usepackage{subfigure}

      \theoremstyle{plain}
      \newtheorem{theorem}{Theorem}[section]
      \newtheorem{lemma}[theorem]{Lemma}
            \newtheorem{problem}[theorem]{Problem}
      \newtheorem{corollary}[theorem]{Corollary}
            \newtheorem{observation}[theorem]{Observation}

      \theoremstyle{definition}

      \theoremstyle{remark}

\def\twr{\mbox{\rm twr}}
\def\conv{\mbox{\rm conv}}

\title{A Ramsey-type result for geometric $\ell$-hypergraphs}
\author{Dhruv Mubayi\thanks{Department of Mathematics, Statistics, and Computer Science, University of Illinois, Chicago, IL, 60607 USA.  Research partially supported by NSF grant DMS-0969092. Email: {\tt mubayi@math.uic.edu}} \and Andrew Suk\thanks{Massachusetts Institute of Technology, Cambridge, MA. Supported by an NSF Postdoctoral
Fellowship and by Swiss National Science Foundation Grant 200021-125287/1. Email: {\tt asuk@math.mit.edu}.}}

\begin{document}
\maketitle

\medskip

\begin{abstract}

Let $n \ge \ell \ge 2$ and $q \ge 2$. We consider the minimum $N$ such that whenever we have $N$ points in the plane in general position and the $\ell$-subsets of these points are colored with $q$ colors, there is a subset $S$ of $n$ points all of whose $\ell$-subsets have the same color and furthermore $S$ is in convex position. This combines two classical areas of intense study over the last 75 years: the Ramsey problem for hypergraphs and the Erd\H os-Szekeres theorem on convex configurations in the plane.  For the special case $\ell = 2$, we establish a single exponential bound on the minimum $N$, such that every complete $N$-vertex geometric graph whose edges are colored with $q$ colors, yields a monochromatic convex geometric graph on $n$ vertices.

 For fixed $\ell \ge 2$ and $q \ge 4$, our results determine the correct exponential tower growth rate for $N$ as a function of $n$, similar to the usual hypergraph Ramsey problem, even though we require our monochromatic set to be in convex position.  Our results also apply to the case of $\ell=3$ and $q=2$  by using a geometric variation of the stepping up lemma of  Erd\H os and Hajnal.
This is in contrast to the fact that the upper and lower bounds  for the usual 3-uniform hypergraph Ramsey problem for two colors differ by one exponential in the tower.
\end{abstract}

\section{Introduction}

The classic 1935 paper of Erd\H os and Szekeres \cite{es} entitled \emph{A Combinatorial Problem in Geometry} was a starting point of a very rich discipline within combinatorics: Ramsey theory (see, e.g., \cite{graham}).  The term \emph{Ramsey theory} refers to a large body of deep results in mathematics which have a common theme: ``Every large system contains a large well-organized subsystem."  Motivated by the observation that any five points in the plane in general position\footnote{A planar point set $P$ is in \emph{general position}, if no three members are collinear.} must contain four members in convex position, Esther Klein asked the following.

\begin{problem}
\label{esn}
For every integer $n\geq 2$, determine the minimum $f(n)$, such that any set of $f(n)$ points in the plane in general position, contains $n$ members in convex position.
\end{problem}

\noindent Celebrated results of Erd\H os and Szekeres \cite{es,es2} imply that

\begin{equation}
\label{f}
2^{n-1} + 1 \leq f(n) \leq {2n - 4\choose n-2} \leq 2^{2n(1 - o(1))}.
\end{equation}

 \noindent They conjectured that $f(n) = 2^{n-1} + 1$, and Erd\H os offered a \$500 reward for a proof.  Despite much attention over the last 75 years, the constant factors in the exponents have not been improved.

  In the same paper \cite{es}, Erd\H os and Szekeres gave another proof of a classic result due to Ramsey \cite{ramsey} on hypergraphs.  An $\ell$-uniform hypergraph $H$ is a pair $(V,E)$, where $V$ is the vertex set and $E\subset {V\choose \ell}$ is the set of edges.  We denote $K^{\ell}_n = (V,E)$ to be the complete $\ell$-uniform hypergraph on an $n$-element set $V$, where $E = {V\choose \ell}$.  When $\ell = 2$, we write $K^2_n = K_n$.  Motivated by obtaining good quantitative bounds on $f(n)$, Erd\H os and Szekeres looked at the following problem.

 \begin{problem}
 \label{rn}
 For every integer $n\geq 2$, determine the minimum integer $r(K_n,K_n)$, such that any two-coloring on the edges of a complete graph $G$ on $r(K_n,K_n)$ vertices, yields a monochromatic copy of $K_n$.
 \end{problem}

\noindent Erd\H os and Szekeres \cite{es} showed that $r(K_n,K_n) \leq 2^{2n}$.  In \cite{erdos2}, Erd\H{o}s gave a construction showing that $r(K_n,K_n)>2^{n/2}$.  Despite much attention over the last 65 years, the constant factors in the exponents have not been improved.

Generalizing Problem \ref{rn} to $q$-colors and $\ell$-uniform hypergraphs has also be studied extensively.  Let $r(K^{\ell}_n;q)$ be the least integer $N$, such that any $q$-coloring on the edges of a complete $N$-vertex $\ell$-uniform hypergraph $H$, yields a monochromatic copy of $K^{\ell}_n$.  We will also write

$$r(K^{\ell}_n;q)  = r(\underbrace{K^{\ell}_n,K^{\ell}_n,...,K^{\ell}_n}_{q\textnormal{ times}}).$$

\noindent Erd\H os, Hajnal, and Rado \cite{erdos3, rado} showed that there are positive constants $c$ and $c'$ such that

\begin{equation}
\label{r3}
2^{cn^2} < r(K^3_n,K^3_n) < 2^{2^{c'n}}.
\end{equation}

\noindent
They also conjectured that $r(K^3_n,K^3_n)  > 2^{2^{cn}}$ for some constant $c > 0$, and Erd\H os offered a \$500 reward for a proof.  For $\ell\geq 4$, there is also a difference of one exponential between the known upper and lower bounds for $r(K^{\ell}_n,K^{\ell}_n)$, namely,

\begin{equation}
\label{rl}
\twr_{\ell-1}(cn^2) \leq  r(K^{\ell}_n,K^{\ell}_n) \leq \twr_{\ell}(c'n),
\end{equation}

\noindent where $c$ and $c'$ depend only on $\ell$, and the tower function $\twr_{\ell}(x)$ is defined by $\twr_1(x) = x$ and $\twr_{i + 1} = 2^{\twr_i(x)}$.   As Erd\H os and Rado have shown \cite{rado}, the upper bound in equation (\ref{rl}) easily generalizes to $q$ colors, implying that $r(K^{\ell}_n;q) \leq \twr_{\ell}(c'n)$, where $c' = c'(\ell,q)$.  On the other hand, for $q\geq 4$ colors, Erd\H os and Hajnal (see \cite{graham}) showed that $r(K^{\ell}_n;q)$ does indeed grow as a $\ell$-fold exponential tower in $n$, proving that $r(K^{\ell}_n;q) = \twr_{\ell}(\Theta(n))$.  For $q = 3$ colors, Conlon, Fox, and Sudakov \cite{conlon} modified the construction of Erd\H os and Hajnal to show that $r(K^{\ell}_n,K^{\ell}_n,K^{\ell}_n,) > \twr_{\ell}(c\log^2n)$.

Interestingly, both Problems \ref{esn} and \ref{rn} can be asked simultaneously for geometric graphs, and a similar-type problem can be asked for geometric $\ell$-hypergraphs.  A \emph{geometric $\ell$-hypergraph} $H$ \emph{in the plane} is a pair $(V,E)$, where $V$ is a set of points in the plane in general position, and $E\subset {V\choose \ell}$ is a collection of $\ell$-tuples from $V$.  When $\ell =2$ ($\ell = 3$), edges are represented by straight line segments (triangles) induced by the corresponding vertices.  The sets $V$ and $E$ are called the \emph{vertex set} and \emph{edge set} of $H$, respectively.  A geometric hypergraph $H$ is \emph{convex}, if its vertices are in convex position.

Geometric graphs ($\ell = 2$) have been studied extensively, due to their wide range of applications in combinatorial and computational geometry (see \cite{pach}, \cite{kar1,kar2}).  Complete convex geometric graphs are very well understood, and are some of the most \emph{well-organized} geometric graphs (if not the most).  Many long standing problems on complete geometric graphs, such as its crossing number \cite{ab}, number of halving-edges \cite{wagner}, and size of crossing families \cite{aronov}, become trivial when its vertices are in convex position.  There has also been a lot of research on geometric 3-hypergraphs in the plane, due to its connection to the  \emph{$k$-set problem} in $\mathbb{R}^3$ (see \cite{mat},\cite{suk},\cite{dey}).  In this paper, we study the following problem which combines Problems \ref{esn} and \ref{rn}.

\begin{problem}
\label{georam}
Determine the minimum integer  $g(K_n^{\ell};q)$, such that any $q$-coloring on the edges of a complete geometric $\ell$-hypergraph $H$ on $g(K_n^{\ell};q)$ vertices, yields a monochromatic convex $\ell$-hypergraph on $n$ vertices.

\end{problem}

Another chromatic variant of the Erd\H os-Szekeres convex polygon problem was studied by Devillers et al.~\cite{dev}, where they considered colored points in the plane rather than colored edges.

We will also write

 $$g(K_n^{\ell};q) = g(\underbrace{K_n^{\ell},...,K_n^{\ell}}_{q\textnormal{ times}}).$$

\noindent Clearly we have $g(K_n^{\ell};q) \geq \max\{r(K^{\ell}_n;q),f(n)\}$.  An easy observation shows that by combining equations (\ref{f}) and (\ref{rl}), we also have

$$g(K_n^{\ell};q)  \leq f(r(K^{\ell}_n;q)) \leq \twr_{\ell + 1}(cn),$$

\noindent where $c = c(\ell,q)$.  Our main results are the following two exponential improvements on the upper bound of $g(K^{\ell}_n;q)$.

\begin{theorem}
\label{main}

For geometric graphs, we have

$$2^{q(n-1)} < g(K_n;q)  \leq 2^{8qn^2\log (qn)}.$$

\end{theorem}

The argument used in the proof of Theorem \ref{main} above extends easily to hypergraphs, and for each fixed $\ell\ge 3$ it gives the bound $g(K^{\ell}_n;q) < \twr_{\ell}(O(n^2))$. David Conlon pointed out to us that one can improve this slightly as follows.

\begin{theorem}
\label{main2}
For geometric $\ell$-hypergraphs, when $\ell \geq 3$ and fixed, we have

$$g(K^{\ell}_n;q)  \leq \twr_{\ell}(cn),$$
\noindent where $c = O(q\log q)$.

\end{theorem}

\noindent By combining Theorems \ref{main}, \ref{main2}, and the fact that $g(K_n^{\ell};q) \geq r(K^{\ell}_n;q)$, we have the following corollary.

\begin{corollary}
For fixed $\ell$ and $q\geq 4$, we have $g(K_n^{\ell};q) = \twr_{\ell}(\Theta(n))$.

\end{corollary}

As mentioned above, there is an exponential difference between the known upper and lower bounds for $r(K^3_n,K^3_n)$.  Hence, for two-colorings on geometric 3-hypergraphs in the plane, equation (\ref{r3}) implies

$$g(K_n^3,K_n^3) \geq r(K^3_n,K^3_n) \geq 2^{cn^2}.$$

\noindent Our next result establishes an exponential improvement in the lower bound of $g(K_n^3,K_n^3)$, showing that $g(K_n^3,K_n^3)$ does indeed grow as a 3-fold exponential tower in a power of $n$.  One noteworthy aspect of this lower bound is that the construction is a geometric version of the famous stepping up lemma of Erd\H os and Hajnal \cite{erdos3} for sets.  The lemma produces a $q'$-coloring $\chi'$ of ${[2^n] \choose \ell+1}$ from a $q$-coloring $\chi$ of ${[n] \choose \ell}$ for appropriate $q'>q$, where the largest monochromatic clique of $(\ell+1)$-sets under $\chi'$ is not too much larger than the largest monochromatic clique of $\ell$-sets under $\chi$  (see \cite{graham} for more details about the stepping up lemma).  While it is a major open problem to apply this method to $r(K^3_n,K^3_n)$ and improve the lower bound in equation (\ref{r3}), we are able to achieve this in the geometric setting as shown below.

\begin{theorem}
\label{lmain}
For geometric 3-hypergraphs in the plane, we have
$$g(K_n^3,K_n^3) \geq 2^{2^{cn}},$$

\noindent where $c$ is an absolute constant.  In particular, $g(K_n^3,K_n^3) = \twr_3(\Theta(n)).$

\end{theorem}

We systemically omit floor and ceiling signs whenever they are not crucial for the sake of clarity of presentation.  All logarithms are in base 2.

\section{Proof of Theorems \ref{main} and \ref{main2}}

Before proving Theorems \ref{main} and \ref{main2}, we will first define some notation.  Let $V = \{p_1,...,p_N\}$ be a set of $N$ points in the plane in general position ordered from left to right according to $x$-coordinate, that is, for $p_i = (x_i,y_i) \in \mathbb{R}^2$, we have $x_i < x_{i + 1}$ for all $i$.  For $i_1 < \cdots < i_t$, we say that $X = (p_{i_1},...,p_{i_t})$ forms an \emph{$t$-cup} (\emph{$t$-cap}) if $X$ is in convex position and its convex hull is bounded above (below) by a single edge.  See Figure~\ref{cupcap}. When $t = 3$, we will just say $X$ is a cup or a cap.

\begin{figure}[h]
\begin{center}
\includegraphics[width=280pt]{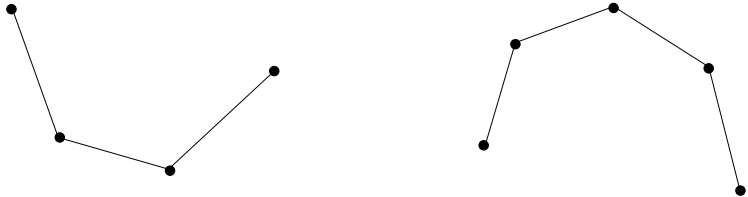}
  \caption{A 4-cup and a 5-cap.}\label{cupcap}
 \end{center}
\end{figure}

\noindent\textbf{Proof of Theorem \ref{main}.}  We first prove the upper bound.  Let $G = (V,E)$ be a complete geometric graph on $N = 2^{8qn^2\log (qn)}$ vertices, such that the vertices $V = \{v_1,...,v_N\}$ are ordered from left to right according to $x$-coordinate.  Let $\chi$ be a $q$-coloring on the edge set $E$.  We will recursively construct a sequence of vertices $p_1,...,p_t$ from $V$ and a subset $S_t \subset V$, where $t = 0,1,...,qn^2$ (when $t =0$ there are no vertices in the sequence), such that the following holds.

\begin{enumerate}

\item for any vertex $p_i$, all pairs $(p_i,p)$ where $p\in \{p_j: j > i\}\cup S_t$ have the same color, which we denote by $\chi'(p_i)$,

\item for every pair of vertices $p_i$ and $p_j$, where $i < j$, either $(p_i,p_j,p)$ is a cap for all $p \in \{p_k: k > j\}\cup S_t$, or $(p_i,p_j,p)$ is a cup for all $p \in \{p_k: k > j\}\cup S_t$,

    \item the set of points $S_t$ lies to the right of the point $p_{t}$, and

    \item $|S_t| \geq \frac{N}{q^tt!} - t$.

\end{enumerate}

We start with no vertices in the sequence, and set $S_0 = V$.  After obtaining vertices $\{p_1,...,p_t\}$ and $S_t$, we define $p_{t+1}$ and $S_{t+1}$ as follows.  Let $p_{t+1} = (x_{t+1},y_{t + 1}) \in \mathbb{R}^2$ be the smallest indexed element in $S_t$ (the left-most point), and let $H$ be the right half-plane $x > x_{t+1}$.  We define $t$ lines $l_1,...l_t$ such that $l_i$ is the line going through points $p_i,p_{t+1}$.  Notice that the arrangement $\cup_{i = 1}^t l_i$ partitions the right half-plane $H$ into $t+1$ cells.  See Figure~\ref{delta}. Since $V$ is in general position, by the pigeonhole principle, there exists a cell $\Delta\subset H$ that contains at least $(|S_t|-1)/(t+1)$ points of $S_t$.

\begin{figure}[h]
\begin{center}
\includegraphics[width=200pt]{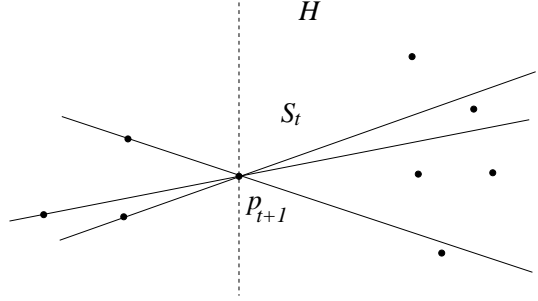}
  \caption{Lines partitioning the half-plane $H$.}\label{delta}
 \end{center}
\end{figure}

Let us call two elements $v'_1,v'_2 \in \Delta\cap S_t$ \emph{equivalent} if $\chi(p_{t+1},v'_1) = \chi(p_{t + 1},v'_2)$.  Hence, there are at most $q$ equivalence classes.  By setting $S_{t + 1}$ to be the largest of those classes, we have the recursive formula

$$|S_{t+ 1}| \geq \frac{|S_t| -1 }{(t + 1)q}.$$

\noindent Substituting in the lower bound on $|S_t|$, we obtain the desired bound

$$|S_{t +1}| \geq \frac{N}{(t+1)!q^{t+1}} - (t+1).$$
This shows that we can construct the sequence $p_1,\ldots,p_{t+1}$ and the set $S_{t+1}$ with the desired properties.  For $N = 2^{8qn^2\log (qn)}$, we have

\begin{equation}
\label{calc}
|S_{qn^2}| \geq  \frac{2^{8qn^2\log (qn)}}{(qn^2)!q^{qn^2}} - qn^2 \geq 1.
\end{equation}

\noindent Hence, $P_1=\{p_1,...,p_{qn^2}\}$ is well defined.  Since $\chi'$ is a $q$-coloring on $P_1$, by the pigeonhole principle, there exists a subset $P_2 \subset P_1$ such that $|P_2| = n^2$, and every vertex has the same color.  By construction of $P_2$, every pair in $P_2$ has the same color.  Hence these vertices induce a monochromatic geometric graph.

Now let $P_2 = \{p'_1,...,p'_{n^2}\}$.  We define partial orders $\prec_1,\prec_2$ on $P_2$, where $p'_i \prec_1 p'_j$ $(p'_i \prec_2 p'_j)$ if and only if $i < j$ and the set of points $P_2\setminus\{p'_1,...,p'_j\}$ lie above (below) the line going through points $p'_i$ and $p'_j$.  See Figure \ref{ex}.  By construction of $P_2$, $\prec_1,\prec_2$ are indeed partial orders and every two elements in $P_2$ are comparable by either $\prec_1$ or $\prec_2$.  By a corollary to Dilworth's Theorem \cite{dil} (see also Theorem 1.1 in \cite{foxsuk}), there exists a chain $p^{\ast}_1,...,p^{\ast}_n$ of length $n$ with respect to one of the partial orders.  Hence $(p^{\ast}_1,...,p^{\ast}_n)$ forms either an $n$-cap or an $n$-cup.  Therefore, these vertices induce a monochromatic convex geometric graph.

 \begin{figure}[h]
  \centering
    \subfigure[Example of $p'_i\prec_1p'_j$.]{\includegraphics[width=.4\textwidth]{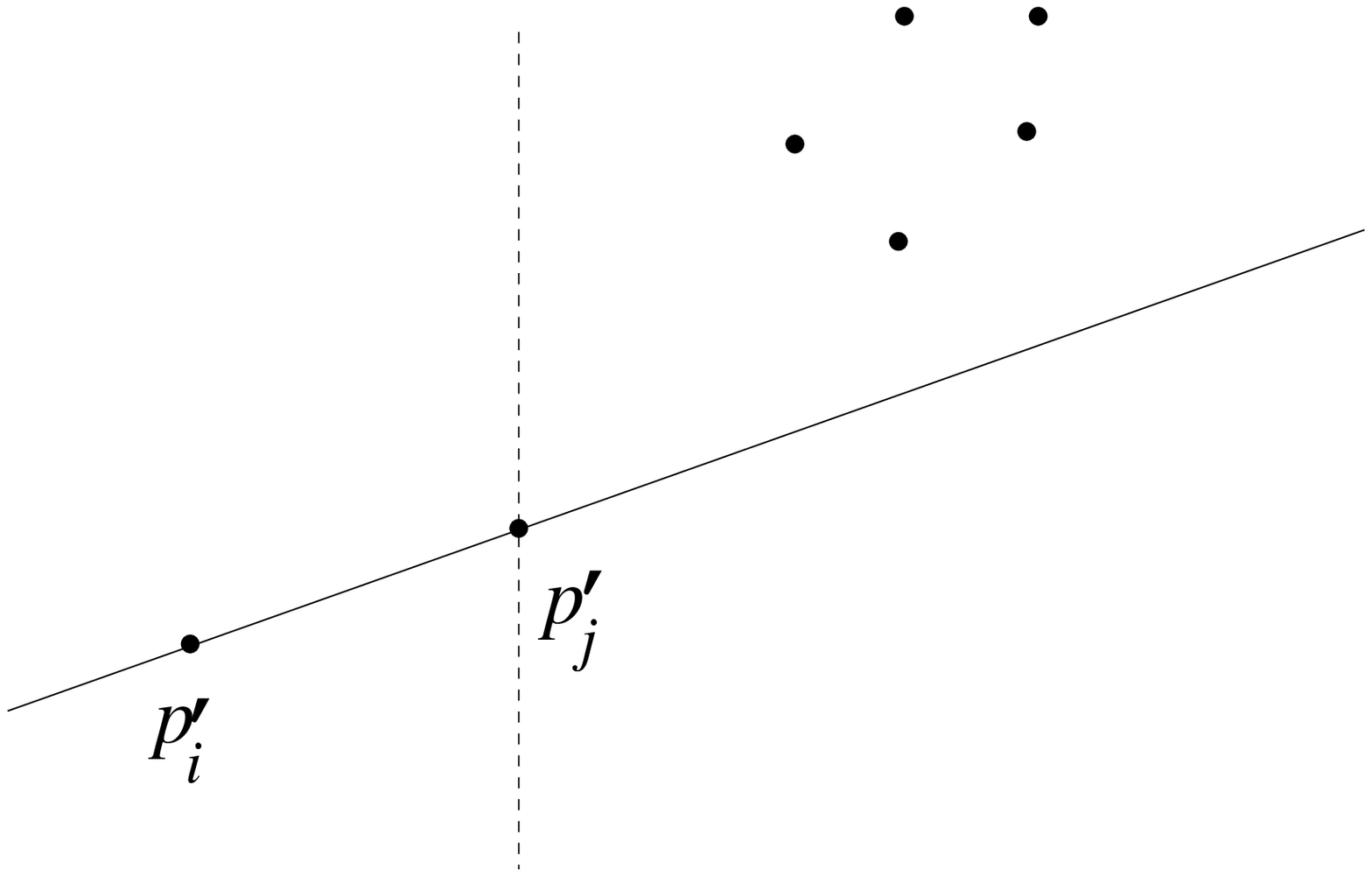}} \hspace{1cm}
        \subfigure[Example of $p'_i\prec_2p'_j$.]{\includegraphics[width=.4\textwidth]{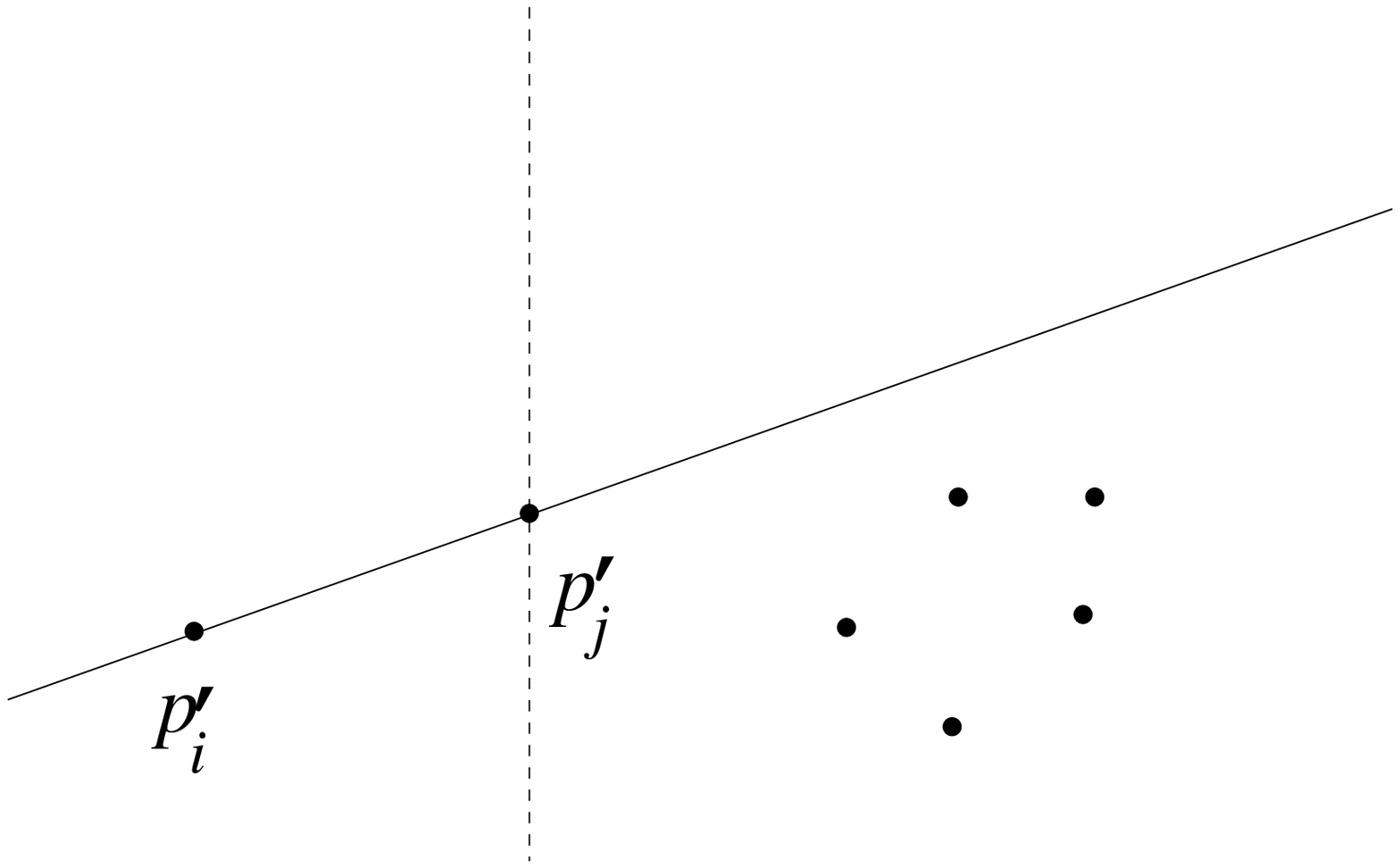}}
        \caption{Partial orders $\prec_1,\prec_2$.}  \label{ex}
\end{figure}

\medskip

For the lower bound, we proceed by induction on $q$.  The base case $q = 1$ follows by taking the complete geometric graph on $2^{n-1}$ vertices, whose vertex set does not have $n$ members in convex position.  This is possible by the construction of Erd\H os and Szekeres \cite{es2}.  Let $G_0$ denote this geometric graph. For $q > 1$, we inductively construct a complete geometric graph $G = (V,E)$ on $2^{(q-1)(n-1)}$ vertices, and a coloring $\chi:E \rightarrow \{1,2,...,q-1\}$ on the edges of $G$, such that $G$ does not contain a monochromatic convex geometric graph on $n$ vertices.  Now we replace each vertex $v_i \in G$ with a small enough copy\footnote{Obtained by an affine transformation.} of $G_0$, which we will denote as $G_i$, where all edges in $G_i$ are colored with the color $q$, and all edges between $G_i$ and $G_j$ have color $\chi(v_iv_j)$.  Then we have a complete geometric graph $G'$ on

$$2^{(q-1)(n-1)}2^{n-1} = 2^{q(n-1)}$$

\noindent vertices, such that $G'$ does not contain a monochromatic convex graph on $n$ vertices.

$\hfill\square$

\medskip

By following the proof above, one can show that $g(K^{\ell}_n;q) \leq \twr_{\ell}(O(n^2))$.  However, the following short argument due to David Conlon gives a  better bound. The proof uses an old idea of M. Tarsi (see \cite{mat} Chapter 3) that gave an upper bound on $f(n)$.

\begin{lemma}
\label{david}
For geometric 3-hypergraphs, we have $g(K^{3}_n;q) \leq r(K^{3}_n;2q) \leq 2^{2^{cn}}$, where $c = O(q\log q)$.
\end{lemma}
\begin{proof}
Let $H = (V,E)$ be a complete geometric 3-hypergraph on $N = r(K^{3}_n;2q)$ vertices, and let $\chi$ be a $q$ coloring on the edges of $H$.  By fixing an ordering on the vertices $V = \{v_1,...,v_N\}$, we say that a triple $(v_{i_1},v_{i_2},v_{i_3})$, $i_1 < i_2 < i_3$, has a \emph{clockwise (counterclockwise) orientation}, if $v_{i_1}, v_{i_2}, v_{i_3}$ appear in clockwise (counterclockwise) order along the boundary of $conv(v_{i_1}\cup v_{i_2} \cup v_{i_3})$. Hence by Ramsey's theorem, there are $n$ points from $V$ for which every triple has the same color and the same orientation.  As observed by Tarsi (see Theorem 3.8 in \cite{comb}), these vertices must be in convex position.
\end{proof}

\begin{lemma}
\label{david2}
For $\ell\geq 4$ and $n \geq 4^{\ell}$, we have $g(K^{\ell}_n;q) \leq r(K^{\ell}_n;q +1 ) \leq \twr_{\ell}(cn)$, where $c = O(q\log q)$.

\end{lemma}

\begin{proof}
Let $H = (V,E)$ be a complete geometric $\ell$-hypergraph on $N = r(K^{\ell}_n;q +1 )$ vertices, and let $\chi$ be a $q$ coloring on the $\ell$-tuples of $V$ with colors $1,2,...,q$.  Now if an $\ell$-tuple from $V$ is not in convex position, we change its color to the new color $q+1$.  By Ramsey's theorem, there is a set $S\subset V$ of $n$ points for which every $\ell$-tuple has the same color.  Since $n \geq 4^{\ell}$, by the Erd\H os-Szekeres Theorem, $S$ contains $\ell$ members in convex position.  Hence, every $\ell$-tuple in $S$ is in convex position, and has the same color which is not the new color $q+1$.  Therefore $S$ induces a monochromatic convex geometric $\ell$-hypergraph.

\end{proof}

Theorem \ref{main2} now follows by combining Lemma \ref{david} and \ref{david2}.

\section{A lower bound construction for geometric 3-hypergraphs}

In this section, we will prove Theorem \ref{lmain}, which follows immediately from the following lemma.

\begin{lemma}
For sufficiently large $n$, there exists a complete geometric 3-hypergraph $H = (V,E)$ in the plane with $2^{2^{\lfloor n/2\rfloor}}$ vertices, and a two-coloring $\chi'$ on the edge set $E$, such that $H$ does not contain a monochromatic convex 3-hypergraph on $2n$ vertices.
\end{lemma}

\begin{proof}
Let $G$ be the complete graph on $2^{n/2}$ vertices, where $V(G) = \{1,...,2^{n/2}\}$, and let $\chi$ be a red-blue coloring on the edges of $G$ such that $G$ does not contain a monochromatic complete subgraph on $n$ vertices.  Such a graph does indeed exist by a result of Erd\H os \cite{erdos2}, who showed that $r(K_n,K_n) > 2^{n/2}$.  We will use $G$ and $\chi$ to construct a complete geometric 3-hypergraph $H$ on $2^{2^{n/2}}$ vertices, and a coloring $\chi'$ on the edges of $H$, with the desired properties.

Set $M = 2^{n/2}$.  We will recursively construct a point set $P_t$ of $2^t$ points in the plane as follows.  Let $P_1$ be a set of two points in the plane with distinct $x$-coordinates.  After obtaining the point set $P_t$, we define $P_{t + 1}$ follows.  We inductively construct two copies of $P_t$, $L = P_t$ and $R = P_t$, and place $L$ to the left of $R$ such that all lines determined by pairs of points in $L$ go below $R$ and all lines determined by pairs of points of $R$ go above $L$.  Then set $P_{t+1} = L\cup R$.  See Figure \ref{pr}.

\begin{figure}[h]
\begin{center}
\includegraphics[width=200pt]{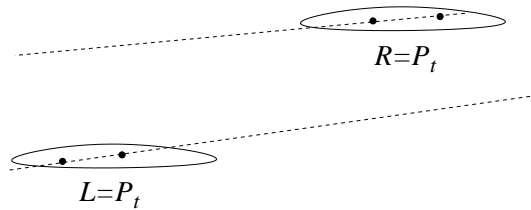}
  \caption{Constructing $P_{t+1}$ from $P_t$.}\label{pr}
 \end{center}
\end{figure}

Let $P_M = \{p_1,...,p_{2^M}\}$ be the set of $2^M$ points in the plane, ordered by increasing $x$-coordinate, from our construction.  Notice that $P_M$ contains $2^{M - t}$ disjoint copies of $P_t$.  For $i< j$, we define

$$\delta(p_i,p_j) = \max\{t: \textnormal{$p_i,p_j$ lies inside a copy of $P_{t} = L\cup R $, and $p_i\in L, p_j\in R$}\}.$$

\noindent Notice that

\begin{description}

\item[Property A:] $\delta(p_i,p_j) \not = \delta(p_j,p_k)$ for every triple $i < j < k$,

\item[Property B:] for $i_1 < \cdots < i_n$, $\delta(p_{i_1},p_{i_n}) = \max_{1 \leq j \leq n-1}\delta(p_{i_j},p_{i_j + 1})$.

\end{description}

Now we define a red-blue coloring $\chi'$ on the triples of $P_M$ as follows.  For $i < j < k$,

$$\chi'(p_i,p_j,p_k) = \chi(\delta(p_i,p_j),\delta(p_j,p_k)).$$

\noindent Now we claim that the geometric 3-hypergraph $H = (P_M,E)$ does not contain a monochromatic convex $3$-hypergraph on $2n$ vertices.  For sake of contradiction, let $S = \{q_1,...,q_{2n}\}$ be a set of $2n$ points from $P_M$, ordered by increasing $x$-coordinate, that induces a red convex 3-hypergraph.  Set $\delta_i = \delta(q_i,q_{i + 1})$.

\medskip
\noindent \emph{Case 1.}  Suppose that there exists a $j$ such that $\delta_j,\delta_{j + 1},...,\delta_{j + n - 1}$ forms a monotone sequence. First assume that

$$\delta_j > \delta_{j + 1} > \cdots >\delta_{j + n - 1}.$$

\noindent Since $G$ does not contain a red complete subgraph on $n$ vertices, there exists a pair $j\leq i_1 < i_2 \leq j+n - 1$ such that $(\delta_{i_1},\delta_{i_2})$ is blue. But then the triple $(q_{i_1},q_{i_2},q_{i_2 + 1})$ is blue, a contradiction. Indeed, by Property~B,

$$\delta(q_{i_1},q_{i_{2}}) = \delta(q_{i_1},q_{i_{1} + 1}) = \delta_{i_1}.$$

\noindent
Therefore, since $\delta_{i_1} > \delta_{i_2}$ and  $(\delta_{i_1},\delta_{i_2})$ is blue, the triple $(q_{i_1},q_{i_2},q_{i_2 + 1})$ must also be blue.  A similar argument holds if $\delta_j < \delta_{j + 1} < \cdots <\delta_{j + n - 1}.$

\medskip

\noindent \emph{Case 2.}  Suppose we are not in Case 1.  For $2 \leq i \leq 2n$, we say that $i$ is a {\it local minimum} if $\delta_{i-1}>\delta_i<\delta_{i+1}$, a {\it local maximum} if $\delta_{i-1}<\delta_i>\delta_{i+1}$, and a \emph{local extremum} if it is either a local minimum or a local maximum.  This is well defined by Property A.

\begin{observation}
\label{notconv}
 For $2 \leq i \leq 2n$, $i$ is never a local minimum.

\end{observation}

\begin{proof}
Suppose $\delta_{i-1}>\delta_i<\delta_{i+1}$ for some $i$, and suppose that $\delta_{i-1} \geq \delta_{i + 1}$.  We claim that $q_{i+1} \in \conv(q_{i-1},q_{i},q_{i + 2})$.  Indeed, since $\delta_{i-1} \geq \delta_{i + 1} > \delta_i$, this implies that $q_{i-1},q_{i},q_{i + 1},q_{i+2}$ lies inside a copy of  $P_{\delta_{i-1}} = L\cup R$, where $q_{i-1} \in L$ and $q_i,q_{i +1},q_{i + 2} \in R$.  Since $\delta_{i + 1} > \delta_i$, this implies that $q_i,q_{i + 1},q_{i + 2}$ lie inside a copy $P_{\delta_{i+1}} = L'\cup R' \subset R$, where $q_{i},q_{i + 1} \in L'$ and $q_{i + 2} \in R'$.

Notice that all lines determined by $q_{i},q_{i+1},q_{i+2}$ go above the point $q_{i-1}$.  Therefore $q_{i+1}$ must lie above the line that goes through the points $q_{i-1},q_{i + 2}$, and furthermore, $q_{i + 1}$ must lie below the line that goes through the points $q_{i-1},q_i$.  Since $\delta_{i + 1} > \delta_i$, the line through $q_{i},q_{i + 1}$ must go below the point $q_{i + 2}$, and therefore $q_{i + 1} \in \conv(q_{i-1},q_{i},q_{i + 2})$.  See Figure \ref{conv4d}.  If $\delta_{i-1} < \delta_{i + 1}$, then a similar argument shows that $q_{i} \in \conv(q_{i-1},q_{i + 1},q_{i + 2})$.

\end{proof}

\begin{figure}
\begin{center}
\includegraphics[width=200pt]{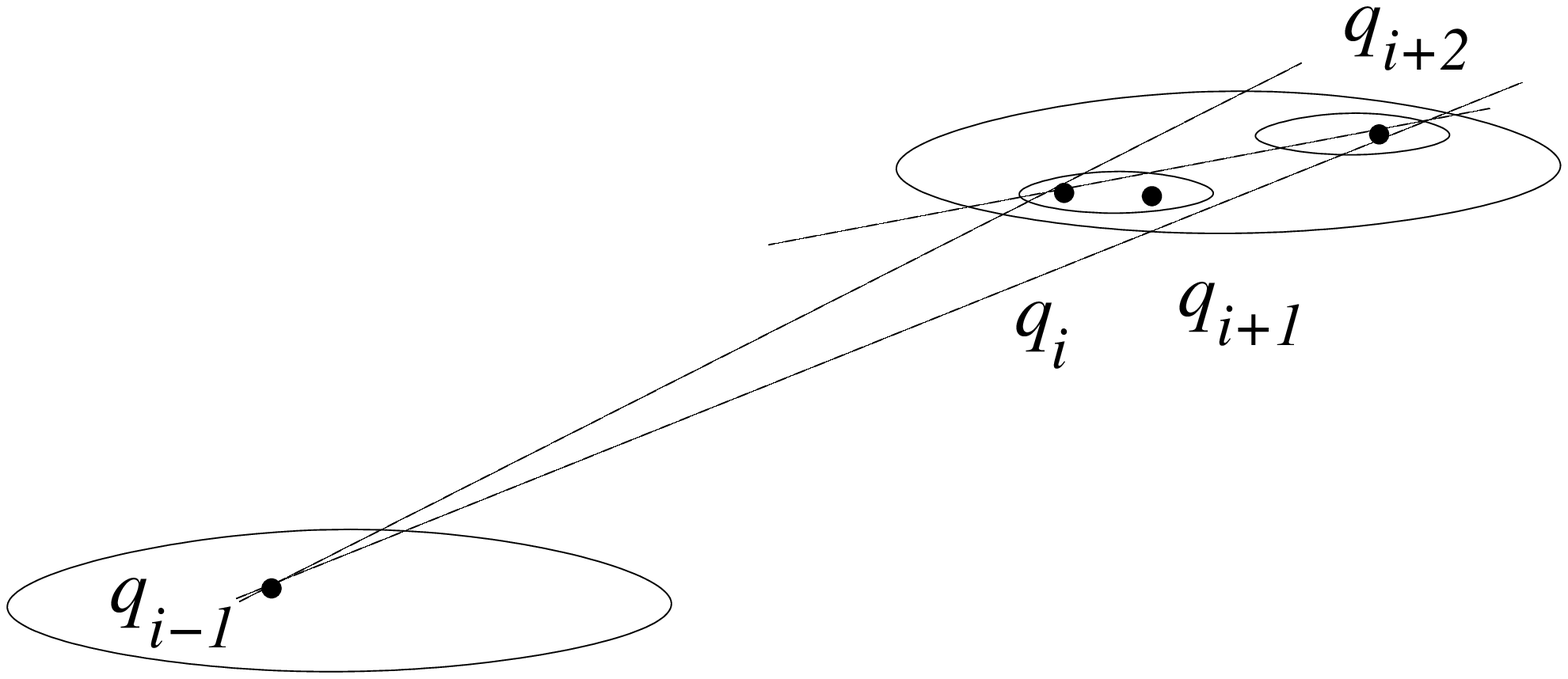}
  \caption{Point $q_{i + 1} \in\conv(q_{i-1},q_{i},q_{i + 2})$ .}\label{conv4d}
 \end{center}
\end{figure}

Since $\delta_1,\ldots,\delta_{2n}$ does not have a monotone subsequence of length $n$, it must have at least two local extrema.  Since between any two local maximums there must be a local minimum, we have a contradiction by Observation \ref{notconv}.  This completes the proof.

\end{proof}

\section{Concluding remarks}

For $q \geq 4$ colors and $\ell \geq 3$, we showed that $g(K_n^{\ell};q) = \twr_{\ell}(\Theta(n)).$  Our bounds on $g(K^{\ell}_n;q)$ for $q\leq 3$ can be summarized in the following table.

\medskip

\medskip
\medskip
\begin{center}
    \renewcommand*\arraystretch{2}
 \begin{tabular}{|c|c|c|}
   \hline
    & $q = 2$ & $q = 3$ \\\hline
   $\ell = 2$  & $2^{\Omega(n)} < g(K_n,K_n) \leq 2^{O(n^2\log n)}$  &  $2^{\Omega(n)} < g(K_n;3) \leq 2^{O(n^2\log n)}$      \\[.2cm]
    $ \ell = 3$ &  $g(K^3_n,K^3_n) = 2^{2^{\Theta(n)}}$   &   $ g(K^3_n;3) = 2^{2^{\Theta(n)}}$  \\
   $\ell \geq 4$  & $\twr_{\ell-1}(\Omega(n^2)) \leq g(K^{\ell}_n,K^{\ell}_n)  \leq \twr_{\ell}(O(n))$  &  $\twr_{\ell}(\Omega(\log^2n)) \leq g(K^{\ell}_n;3)  \leq \twr_{\ell}(O(n))$  \\
   \hline
 \end{tabular}
 \end{center}

\medskip
\medskip
\medskip
\noindent \textbf{Off-diagonal.}  The Ramsey number $r(K_s,K_n)$ is the minimum integer $N$ such that every red-blue coloring on the edges of a complete $N$-vertex graph $G$, contains either a red clique of size $s$, or a blue clique of size $n$.  The off-diagonal Ramsey numbers, i.e., $r(K_s,K_n)$ with $s$ fixed and $n$ tending to infinity, have been intensively studied. For example, it is known \cite{AKS,B,BK, Kim} that $R_2(3,n) =\Theta(n^2/\log n)$ and, for fixed $s > 3$,
\begin{equation}\label{offgraph}
c_1(\log n)^{1/(s-2)} \left(\frac{n}{\log n}\right)^{(s+1)/2} \leq r(K_s,K_n) \leq c_2\frac{n^{s-1}}{\log^{s-2}n}.\end{equation}

 Another interesting variant of Problem \ref{georam} is the following off-diagonal version.

\begin{problem}
Determine the minimum integer  $g(K_s,K_n)$, such that any red-blue coloring on the edges of a complete geometric graph $G$ on $g(K_s,K_n)$ vertices, yields either a red convex geometric graph on $s$ vertices, or a blue convex geometric graph on $n$ vertices.
\end{problem}

For fixed $s$, one can show that $g(K_s,K_n)$ grows single exponentially in $n$.  In particular

$$2^{n-1} + 1 \leq g(K_s,K_n) \leq 4^{4^sn}.$$

\noindent  The lower bound follows from the fact that $g(K_s,K_n) \geq f(n)$.  The upper bound follows from the inequalities

  $$g(K_s,K_n) \leq r(K_{4^s},K_{4^n})\leq \left(4^n\right)^{4^s}.$$

  Indeed, by the Erd\H os-Szkeres theorem, if $G$ contains a red-clique of size $4^s$, then there must be a red convex geometric graph on $s$ vertices.  Likewise, If $G$ contains a blue clique of size $4^n$, then there must be a blue convex geometric graph on $n$ vertices.

\medskip

\noindent \textbf{Higher dimensions.}  Generalizing Problem \ref{esn} to higher dimensions has also been studied.  Let $f_d(n)$ be the smallest integer such that any set of at least $f_d(n)$ points in $\mathbb{R}^d$ in general position\footnote{A point set $P$ in $\mathbb{R}^d$ is in general position, if no $d+1$ members lie on a common hyperplane.} contains $n$ members in convex position.  The following upper and lower bounds were obtained by K\'arolyi \cite{k} and K\'arolyi and Valtr \cite{kv} respectively,

$$2^{cn^{1/(d-1)}} \leq f_d(n) \leq {2n - 2d - 1\choose n - d} + d = 2^{2n(1 - o(1))}.$$

A \emph{geometric $\ell$-hypergraph} $H$ \emph{in $d$-space} is a pair $(V,E)$, where $V$ is a set of points in general position in $\mathbb{R}^d$, and $E\subset {V\choose \ell}$ is a collection of $\ell$-tuples from $V$.  When $\ell \leq d+1$, $\ell$-tuples are represented by $(\ell-1)$-dimensional simplices induced by the corresponding vertices.

\begin{problem}
Determine the minimum integer  $g_d(K_n^{\ell};q)$, such that any $q$-coloring on the edges of a complete geometric $\ell$-hypergraph $H$ in $d$-space on $g_{d}(K_n^{\ell};q)$ vertices, yields a monochromatic convex $\ell$-hypergraph on $n$ vertices.

\end{problem}

When $d = 2$, we write $g_{2}(K_n^{\ell};q) = g(K_n^{\ell};q)$.  Clearly $g_{d}(K_n^{\ell};q) \geq \max\{f_d(n), R(K^{\ell}_n;q)\}$.  One can also show that $g_{d}(K_n^{\ell};q) \leq g(K_n^{\ell};q)$.  Indeed, for any complete geometric $\ell$-hypergraph $H = (V,E)$ in $d$-space with a $q$-coloring $\chi$ on $E(H)$, one can obtain a complete geometric $\ell$-hypergraph in the plane $H' = (V',E')$, by projecting $H$ onto a 2-dimensional subspace $L\subset \mathbb{R}^d$ such that $V'$ is in general position in $L$.  Thus we have

$$g_d(K_n;q)\leq g(K_n;q) \leq 2^{cn^2\log n},$$

\noindent where $c= O(q\log q)$, and for $\ell \geq 3$

$$g_d(K^{\ell}_n;q) \leq g(K^{\ell}_n;q)\leq \twr_{\ell}(c'n^2),$$

\noindent  where  $c'=c'(q,\ell)$.

\medskip
\medskip

\noindent \textbf{Acknowledgment.}  We thank David Conlon for showing us an improved version of Theorem \ref{main2}.

\end{document}